\pdfoutput=1
\documentclass[a4paper,11pt]{amsart}

\usepackage[margin=1in,includehead,includefoot]{geometry}
\usepackage[OT4]{fontenc}
\usepackage{amsmath,amsthm,amssymb}
\usepackage{enumitem}
\usepackage{microtype}
\usepackage[pdftitle={Decomposition of multiple packings with subquadratic union complexity},
            pdfauthor={Janos Pach and Bartosz Walczak},
            colorlinks=true,linkcolor=black,citecolor=black,filecolor=black,urlcolor=black]{hyperref}

\renewenvironment{enumerate}{\begin{enumorig}[label=\textup{(\arabic*)}, noitemsep, topsep=1.5mm plus 1.5mm, leftmargin=2.5em]}{\end{enumorig}}

\newtheorem{theorem}{Theorem}
\newtheorem{lemma}[theorem]{Lemma}
\newtheorem{proposition}[theorem]{Proposition}

\def\setN{\mathbb{N}}
\def\setR{\mathbb{R}}

\def\calA{\mathcal{A}}
\def\calB{\mathcal{B}}
\def\calD{\mathcal{D}}
\def\calF{\mathcal{F}}
\def\calI{\mathcal{I}}
\def\calX{\mathcal{X}}
\def\calY{\mathcal{Y}}

\let\leq\leqslant
\let\geq\geqslant
\let\setminus\smallsetminus
\let\Gamma\varGamma
\let\Theta\varTheta

\def\tbinom#1#2{\smash{\textstyle\binom{#1}{#2}}}

\makeatletter
\let\old@setaddresses\@setaddresses
\def\@setaddresses{\bigskip\bgroup\parindent 0pt\let\scshape\relax\old@setaddresses\egroup}
\makeatother

\title{Decomposition~of~multiple~packings~with subquadratic~union~complexity}

\author{J\'anos Pach\and Bartosz Walczak}

\address[J\'anos Pach]{\'Ecole Polytechnique F\'ed\'erale de Lausanne, Switzerland and Alfred R\'enyi Institute of Mathematics, Hungarian Academy of Sciences, Budapest, Hungary}
\email{pach@cims.nyu.edu}

\address[Bartosz Walczak]{Theoretical Computer Science Department, Faculty of Mathematics and Computer Science, Jagiellonian University, Krak\'ow, Poland}
\email{walczak@tcs.uj.edu.pl}

\thanks{A journal version of this paper appeared in \emph{Combin.\ Prob.\ Comput.}\ 25(1):145--153, 2016.}
\thanks{J\'anos Pach was partially supported by NSF grant CCF-08-30272, by OTKA under EuroGIGA projects GraDR and ComPoSe 10-EuroGIGA-OP-003, and by Swiss National Science Foundation grants 200020-144531 and 200021-137574.}
\thanks{Bartosz Walczak was partially supported by Swiss National Science Foundation grant 200020-144531 and by MNiSW grant 884/N-ESF-EuroGIGA/10/2011/0 under EuroGIGA project GraDR\@.}


\begin{document}

\begin{abstract}
Suppose $k$ is a positive integer and $\mathcal{X}$ is a \emph{$k$-fold packing} of the plane by infinitely many arc-connected compact sets, which means that every point of the plane belongs to at most $k$ sets.
Suppose there is a function $f(n)=o(n^2)$ with the property that any $n$ members of $\mathcal{X}$ determine at most $f(n)$ \emph{holes}, which means that the complement of their union has at most $f(n)$ bounded connected components.
We use tools from extremal graph theory and the topological Helly theorem to prove that $\mathcal{X}$ can be decomposed into at most $p$ ($1$-fold) packings, where $p$ is a constant depending only on $k$ and $f$.
\end{abstract}

\maketitle

\section{Introduction}

The notions of multiple packings and coverings were introduced in a geometric setting independently by Harold Davenport and L\'aszl\'o Fejes T\'oth \cite{FeT72}.
In the present note, we will be concerned only with packings.
A \emph{$k$-fold packing} is a family $\calX$ of sets with the property that the intersection of any $k+1$ members of $\calX$ is empty.
A $1$-fold packing is simply called a \emph{packing}.
The problem of determining the maximum density of a $k$-fold packing with congruent copies of a fixed convex body has been extensively studied---see \cite{FTK93} and references therein.
For small values of $k$, it was found that the densest $k$-fold lattice packings in the plane split into $k$ packings \cite{Blu57,DuH72,Hep59}.
The situation gets more complicated for larger values of $k$, but in general a $k$-fold packing by convex bodies that are \emph{fat} (that is, the ratio of their circumradii to their inradii is bounded) can be decomposed into $O(k)$ packings, as was shown by Pach \cite{Pac80}.
A simple but interesting corollary of this fact is that any $k$-fold packing by homothets (uniformly scaled and translated copies) of a convex body in $\setR^d$ splits into at most $c_dk$ packings, where the constant $c_d$ depends only on the dimension.

The problem of decomposing a family of sets into packings can be rephrased as a coloring problem for intersection graphs.
The \emph{intersection graph} of a family $\calX$ of sets is a graph on the vertex set $\calX$ in which two vertices are joined by an edge if and only if the corresponding members of $\calX$ have nonempty intersection.
Thus, a decomposition of $\calX$ into $k$ packings is precisely a proper $k$-coloring of the intersection graph of $\calX$.

If the intersection graph of $\calX$ has clique number at most $k$, then $\calX$ is a $k$-fold packing, but not necessarily the other way around.
For example, the set of the three (closed) sides of a triangle forms a $2$-fold packing, while its intersection graph is $K_3$, so it has clique number~$3$.
However, for axis-parallel boxes in Euclidean space, the two notions coincide: a family is a $k$-fold packing if and only if the clique number of its intersection graph is at most $k$.
Asplund and Gr\"unbaum \cite{AsG60} proved that the intersection graphs of axis-parallel rectangles in $\setR^2$ with clique number $k$ have chromatic number $O(k^2)$.
Equivalently, every $k$-fold packing of the plane by axis-parallel rectangles can be decomposed into $O(k^2)$ packings.
On the other hand, Burling \cite{Bur65} constructed triangle-free intersection graphs of axis-parallel boxes in $\setR^3$ with arbitrarily large chromatic number.
Pawlik, Kozik, Krawczyk, Laso\'n, Micek, Trotter and Walczak \cite{PKK+14,PKK+13} provided similar constructions for straight-line segments and other kinds of geometric sets in the plane with the property that their intersection graphs are triangle-free but can have arbitrarily large chromatic number.
For a survey on coloring geometric intersection graphs, see \cite{Kos04}.

The aim of the present note is to show that for every family $\calF$ of geometric objects of ``small complexity'' (in the sense described later), there exists a function $p(k)$ such that every $k$-fold packing of the plane by members of $\calF$ can be split into $p(k)$ packings.

There are some standard measures of complexity for families of geometric objects, used in bounding the computational complexity of various algorithms in motion planning, computer vision, and geometric transversal theory.
A \emph{simple arc} with respect to a finite family $\calX$ of sets is a Jordan arc whose interior is entirely contained in or disjoint from every set in $\calX$.
The \emph{union boundary complexity} of $\calX$ is the minimum number of simple arcs whose union is the boundary of $\bigcup\calX$.
A related measure of complexity is the number of \emph{holes} in $\bigcup\calX$, that is, bounded arc-connected components of $\setR^2\setminus\bigcup\calX$.
The number of holes is bounded from above by the union boundary complexity.
However, for some families of geometric objects, the number of holes in the union can be much smaller than the union boundary complexity.

We prove that for any fixed $k$, every $k$-fold packing by arc-connected compact sets in the plane, the union of any $n$ of which determines a subquadratic number of holes, can be decomposed into a bounded number of packings.

\begin{theorem}
\label{thm:main}
Let\/ $k\in\setN$, let\/ $f\colon\setN\to\setN$ be a function such that\/ $f(n)=o(n^2)$, and let\/ $\calF$ be an infinite family of arc-connected compact sets in the plane with the property that every finite subfamily\/ $\calX$ of\/ $\calF$ determines at most\/ $f(|\calX|)$ holes.
Then there exists a constant\/ $p=p_f(k)$ such that every\/ $k$-fold packing by members of\/ $\calF$ can be decomposed into\/ $p$ packings.
\end{theorem}

It is enough to prove Theorem~\ref{thm:main} for $k$-fold packings that are finite subfamilies of $\calF$, as then the general statement follows by a standard compactness argument.
Therefore, for the remainder of the paper, every $k$-fold packing that we consider is assumed to be finite.
With this assumption, we will prove the following stronger result.

\begin{theorem}
\label{thm:main2}
Let\/ $k$, $f$, and\/ $\calF$ be the same as in the previous theorem.
Then there exists a constant\/ $p=p_f(k)$ such that the intersection graph of any finite\/ $k$-fold packing by members of\/ $\calF$ has a vertex of degree smaller than\/ $p_f(k)$.
\end{theorem}

One of the earliest results on the union boundary complexity and, hence, on the number of holes was established by Kedem, Livne, Pach and Sharir \cite{KLPS86}.
They proved that the union boundary complexity of every family of $n$ \emph{pseudodiscs}, that is, compact sets in the plane whose boundaries are closed simple curves any two of which share at most two points, is $O(n)$.
Therefore, our theorems imply that any $k$-fold packing by pseudodiscs can be decomposed into a bounded number $p(k)$ of packings.

Matou\v{s}ek, Pach, Sharir, Sifrony and Welzl \cite{MPS+94} showed that families of $n$ fat triangles in the plane determine $O(n)$ holes.
Efrat and Sharir \cite{EfS00} established a near-linear upper bound on the union boundary complexity of families of fat convex sets any two of which share at most a bounded number of boundary points.
Therefore, our results also apply to this case and generalize the planar version of the statement on fat convex sets mentioned in the first paragraph.
For further results on the complexity of various kinds of fat objects, consult \cite{Ber08,Efr05,ERS93,Kre98}.

Our proof of Theorem~\ref{thm:main2} is based on a result due to Fox and Pach \cite{FoP10}, which asserts that the intersection graphs of finite families of arc-connected sets in the plane with no subgraph isomorphic to $K_{t,t}$ have bounded minimum degree, for every $t\geq 2$.
The bound on $p_f(k)$ it gives for a fixed function $f$ is very bad and certainly far from optimal.
For example, if $f(n)=\Theta(n)$, then the bound is more than double exponential in $k$.
Independently of the work in this paper, using the probabilistic sampling technique due to Clarkson and Shor \cite{ClS89} (see also \cite{Sha91}), Micek and Pinchasi \cite{MiP} proved that the intersection graph of a $k$-fold packing by geometric objects with linear union boundary complexity has a vertex of degree $O(k)$.

The assumption that $f(n)=o(n^2)$ is crucial for Theorems \ref{thm:main} and~\ref{thm:main2} to hold.
An $n\times n$ grid of thin horizontal and vertical rectangles forms a $2$-fold packing with $\Theta(n^2)$ holes, and all vertices in the intersection graph of these rectangles have degree $n$.

First, in Section~\ref{sec:pseudodiscs}, we establish our result in a simple special case---for $k$-fold packings by pseudodiscs.
For the proof of Theorem~\ref{thm:main2} in its full generality, we need a technical lemma on the number of holes determined by $2$-fold packings, which is formulated and proved in Section~\ref{sec:euler}.
The proof in the general case is presented in Section~\ref{sec:proof}.

\section{The case of pseudodiscs}
\label{sec:pseudodiscs}

The members of a family of compact sets in the plane are called \emph{pseudodiscs} if the boundary of each set is a simple closed curve and any two of these curves share at most two points.
We first give a short proof of the following assertion.

\begin{proposition}
\label{prop:pseudodiscs}
For every positive integer\/ $k$, there is a constant\/ $p=p(k)$ such that every\/ $k$-fold packing by pseudodiscs has a member that intersects fewer than\/ $p$ other members.
\end{proposition}

For the proof, we use two well-known results: a recent theorem of Fox and Pach \cite{FoP10,FoP14} and the classical topological Helly theorem \cite{Hel30}.
Let $K_{t,t}$ denote a complete bipartite graph with $t$ vertices in each of its parts.
The main idea of the proof of Proposition~\ref{prop:pseudodiscs} is to show that the intersection graph of any $k$-fold packing by pseudodiscs has no subgraph isomorphic to $K_{t,t}$ for $t$ large enough, and then apply the following result.

\begin{theorem}[Fox, Pach \cite{FoP10,FoP14}]
\label{thm:fox-pach}
For any\/ $t\in\setN$, there is a constant\/ $c=c(t)$ with the property that the intersection graph of any finite family of arc-connected sets in the plane with no subgraph isomorphic to\/ $K_{t,t}$ has a vertex of degree smaller than\/ $c$.
Furthermore, this holds with\/ $c(t)=t(\log t)^{\gamma}$ for some absolute constant\/ $\gamma>0$.
\end{theorem}

\begin{theorem}[Helly \cite{Hel30}]
\label{thm:helly}
For any family of pseudodiscs in which every triple has a point in common, all members have a point in common.
\end{theorem}

\begin{proof}[Proof of Proposition \ref{prop:pseudodiscs}]
In view of Theorem~\ref{thm:fox-pach}, it is sufficient to prove that, for $t$ large enough, the intersection graph of the pseudodiscs contains no $K_{t,t}$.
Suppose it does, and consider the $t$ pseudodiscs that form the first vertex class of $K_{t,t}$.
Color each triple of them \emph{red} if they have a point in common and \emph{blue} otherwise.
It follows from Theorem~\ref{thm:helly} that there are no $k+1$ pseudodiscs all of whose triples are red.
Otherwise, they would share a point, contradicting the assumption that the pseudodiscs form a $k$-fold packing.
Thus, if $t$ is large enough, then, by Ramsey's theorem, the first vertex class of $K_{t,t}$ has $9$ pseudodiscs all of whose triples are blue, that is, which form a $2$-fold packing.

To continue, we need an easy observation that the intersection graph $G$ of any $2$-fold packing $\calD$ of the plane by pseudodiscs is planar.
To see this, first note that if there are two nested pseudodiscs in $\calD$, then the inner one can be disregarded, as it has degree $1$ in $G$.
Thus, assume there are no two nested pseudodiscs in $\calD$.
For each pseudodisc $A\in\calD$, choose a point $x_A\in A$ that lies in no other pseudodisc in $\calD$.
Then, for each intersecting pair of pseudodiscs $A,B\in\calD$, connect the points $x_A$ and $x_B$ by an arc that lies in $A\cup B$ but avoids all pseudodiscs in $\calD\setminus\{A,B\}$.
This yields a drawing of $G$ in which every crossing pair of edges shares an endpoint.
For each crossing point, considered one by one, remove from the two crossing edges some very small parts around that point, and reconnect the remaining parts of the edges appropriately in one of the two possible noncrossing ways so as to obtain a new drawing of $G$ with that crossing removed.
After removing all crossings, we are left with a plane drawing of $G$.

It follows that the intersection graph of the $9$ pseudodiscs from the first vertex class of $K_{t,t}$ is planar.
Consequently, it is properly $4$-colorable, so at least $3$ of the $9$ pseudodiscs must be pairwise disjoint.
In the same way, we can choose $3$ pairwise disjoint pseudodiscs from the second vertex class of $K_{t,t}$.
Again, the $6$ pseudodiscs thus chosen form a $2$-fold packing, so their intersection graph is planar, but that graph is $K_{3,3}$.
This is the desired contradiction.
\end{proof}

\section{Lower bound on the number of holes}
\label{sec:euler}

For a set $X\subseteq\setR^2$, let $\Gamma(X)$ denote the family of arc-connected components of $X$, and let $h(X)$ stand for the number of holes in $X$, that is, bounded arc-connected components of $\setR^2\setminus X$.

\begin{lemma}
\label{lem:euler}
Let\/ $X_1,\ldots,X_n$ be (not necessarily distinct) compact sets in the plane that form a\/ $2$-fold packing (the intersection of any three of them is empty).
Let\/ $S$ be the set of points that belong to exactly two sets from\/ $X_1,\ldots,X_n$.
Then
\[h\Bigl(\bigcup_{i=1}^nX_i\Bigr)\geq|\Gamma(S)|-\sum_{i=1}^n|\Gamma(X_i)|+1.\]
\end{lemma}

\begin{proof}
Let $G$ be the bipartite graph with the following vertex and edge sets:
\begin{gather*}
V(G)=\Gamma(S)\cup\bigcup_{i=1}^n\Gamma(X_i\setminus S),\\
E(G)=\Bigl\{(A,B)\in\Gamma(S)\times\Bigl(\bigcup_{i=1}^n\Gamma(X_i\setminus S)\Bigr)\colon\text{$A$ and $B$ touch each other}\Bigr\}.
\end{gather*}
That is, $G$ is the contact graph of the sets in $V(G)$.
We have $\Gamma(X_i\cap S)\subseteq\Gamma(S)$ for every $i$.
Let $G_i$ denote the subgraph of $G$ induced by the vertex set $\Gamma(X_i\cap S)\cup\Gamma(X_i\setminus S)$.
The number of connected components of $G_i$ is exactly $|\Gamma(X_i)|$, and thus $|E(G_i)|\geq|V(G_i)|-|\Gamma(X_i)|$.
Since each edge of $G$ belongs to exactly one of $G_1,\ldots,G_n$ and each arc-connected component of $S$ belongs to exactly two of $X_1,\ldots,X_n$, we have
\[|E(G)|=\sum_{i=1}^n|E(G_i)|\geq\sum_{i=1}^n\bigl(|V(G_i)|-|\Gamma(X_i)|\bigr)=|V(G)|+|\Gamma(S)|-\sum_{i=1}^n|\Gamma(X_i)|.\]
The graph $G$ is planar---its representation as the contact graph of the sets in $V(G)$ yields a plane drawing of $G$ in a way very similar to that described in the proof of Proposition~\ref{prop:pseudodiscs} for pseudodiscs.
Since each inner face of that drawing surrounds a hole of $X_1\cup\cdots\cup X_n$, the number of holes in $X_1\cup\cdots\cup X_n$ is at least the number of inner faces in the drawing.
Therefore, by Euler's formula, we have
\[h\Bigl(\bigcup_{i=1}^nX_i\Bigr)\geq|E(G)|-|V(G)|+1\geq|\Gamma(S)|-\sum_{i=1}^n|\Gamma(X_i)|+1.\qedhere\]
\end{proof}

Note that if $X_1,\ldots,X_n$ and $S$ are as in Lemma~\ref{lem:euler}, then each intersecting pair of sets from $X_1,\ldots,X_n$ gives rise to at least one separate component of $\Gamma(S)$.

\section{Proof of Theorem \ref{thm:main2}}
\label{sec:proof}

Like in the proof of Proposition~\ref{prop:pseudodiscs}, we will show that the intersection graph of any $k$-fold packing by members of $\calF$ has no subgraph isomorphic to $K_{t,t}$ for $t$ large enough, and then apply Theorem~\ref{thm:fox-pach}.

\begin{lemma}
\label{lem:ind1}
Let\/ $k\geq 2$ be a positive integer, $\alpha$ be a positive real, and\/ $f\colon\setN\to\setN$ be a function such that\/ $f(n)=o(n^2)$.
Then there is an integer\/ $M=M_f(k,\alpha)$ with the following property: if
\begin{enumerate}
\item $X_1,\ldots,X_n$ are arc-connected compact sets in the plane that form a\/ $k$-fold packing,
\item\label{item:2} there are at least\/ $\alpha\tbinom{n}{k}$\/ $k$-tuples of sets from\/ $X_1,\ldots,X_n$ with nonempty intersection,
\item\label{item:3} $h(X_{i_1}\cup\cdots\cup X_{i_t})\leq f(t)$ for any choice of\/ $i_1,\ldots,i_t\in\{1,\ldots,n\}$,
\end{enumerate}
then\/ $n<M$.
\end{lemma}

\begin{proof}
By \ref{item:2} and by the pigeonhole principle, there is a $(k-2)$-tuple of sets from $X_1,\ldots,X_n$ that belongs to at least $\alpha\tbinom{n}{k}\tbinom{k}{k-2}/\tbinom{n}{k-2}=\alpha\tbinom{n-k+2}{2}$ $k$-tuples of sets from $X_1,\ldots,X_n$ with nonempty intersection.
Assume without loss of generality that this $(k-2)$-tuple is $X_1,\ldots,X_{k-2}$.
Let $Y_i=X_1\cap\cdots\cap X_{k-2}\cap X_i$ for $k-1\leq i\leq n$.
It follows that there are at least $\alpha\tbinom{n-k+2}{2}$ intersecting pairs of sets from $Y_{k-1},\ldots,Y_n$.
Since $X_1,\ldots,X_n$ form a $k$-fold packing, the $n-k+2$ sets $Y_{k-1},\ldots,Y_n$ form a $2$-fold packing.
Therefore, by Lemma~\ref{lem:euler}, we have
\begin{equation}
h\Bigl(\bigcup_{i=k-1}^nY_i\Bigr)\geq\alpha\tbinom{n-k+2}{2}-\sum_{i=k-1}^n|\Gamma(Y_i)|+1.\tag{$*$}\label{eq}
\end{equation}

We claim that $|\Gamma(Y_i)|\leq\tbinom{k-1}{2}f(2)+1$ for $k-1\leq i\leq n$.
To show this, we assume without loss of generality that $i=k-1$ (that is, $Y_i=X_1\cap\cdots\cap X_{k-1}$) and use induction on $r$ to prove that $|\Gamma(X_1\cap\cdots\cap X_r)|\leq\tbinom{r}{2}f(2)+1$ for $1\leq r\leq k-1$.
Since $X_1$ is arc-connected, we have $|\Gamma(X_1)|=1$.
For $r\geq 2$, we apply Lemma~\ref{lem:euler} to the two sets $X_1\cap\cdots\cap X_{r-1}$ and $X_r$ to get
\[\begin{split}|\Gamma(X_1\cap\cdots\cap X_r)|&\leq h\bigl((X_1\cap\cdots\cap X_{r-1})\cup X_r\bigr)+|\Gamma(X_1\cap\cdots\cap X_{r-1})|+|\Gamma(X_r)|-1\\
&\leq h\bigl((X_1\cup X_r)\cap\cdots\cap(X_{r-1}\cup X_r)\bigr)+\tbinom{r-1}{2}f(2)+1\\
&\leq h(X_1\cup X_r)+\cdots+h(X_{r-1}\cup X_r)+\tbinom{r-1}{2}f(2)+1\\
&\leq(r-1)f(2)+\tbinom{r-1}{2}f(2)+1=\tbinom{r}{2}f(2)+1.\end{split}\]
The second inequality above follows from the induction hypothesis and the assumption that $X_r$ is arc-connected, the third one follows from the fact that every hole in an intersection of sets contains a hole of one of those sets, and the last one follows from \ref{item:3}.
This proves the claim.

The claim and the inequality \eqref{eq} imply that
\[h\Bigl(\bigcup_{i=k-1}^nY_i\Bigr)\geq\alpha\tbinom{n-k+2}{2}-(n-k+2)\bigl(\tbinom{k-1}{2}f(2)+1\bigr)+1.\]
On the other hand, again by the fact that every hole in an intersection of sets contains a hole of one of those sets and by \ref{item:3}, we have
\[h\Bigl(\bigcup_{i=k-1}^nY_i\Bigr)=h\Bigl(X_1\cap\cdots\cap X_{k-2}\cap\Bigl(\bigcup_{i=k-1}^nX_i\Bigr)\Bigr)\leq(k-2)f(1)+f(n-k+2).\]
The two inequalities yield
\[f(n-k+2)\geq\alpha\tbinom{n-k+2}{2}-(n-k+2)\bigl(\tbinom{k-1}{2}f(2)+1\bigr)+1-(k-2)f(1),\]
which cannot hold for arbitrarily large $n$, as $f(n)=o(n^2)$.
This completes the proof.
\end{proof}

We will need an observation due to Katona, Nemetz and Simonovits \cite{KNS64} generalizing Tur\'an's theorem to $k$-uniform hypergraphs.
A \emph{$k$-uniform hypergraph} $H$ consists of a set of \emph{vertices}, denoted by $V(H)$, and a set of \emph{edges}, denoted by $E(H)$, that are $k$-element subsets of $V(H)$.
An \emph{independent set} in $H$ is a subset of $V(H)$ that does not entirely contain any edge of $H$.

\begin{theorem}[Katona, Nemetz, Simonovits \cite{KNS64}]
\label{thm:turan}
For any\/ $k,m\in\setN$, every\/ $k$-uniform hypergraph with\/ $n\geq m$ vertices and fewer than\/ $\tbinom{n}{k}/\tbinom{m}{k}$ edges contains an independent set of size\/ $m$.
\end{theorem}

\noindent
A bound sharper than that of Theorem \ref{thm:turan} has been established by Spencer \cite{Spe72} using a simple probabilistic argument, but we will not need it.

\begin{lemma}
\label{lem:ind2}
Let\/ $k,\ell\in\setN$, and let\/ $f\colon\setN\to\setN$ be a function such that\/ $f(n)=o(n^2)$.
Then there is an integer\/ $N=N_f(k,\ell)$ with the following property: if
\begin{enumerate}
\item $X_1,\ldots,X_n$ are arc-connected compact sets in the plane that form a\/ $k$-fold packing,
\item no\/ $\ell$ sets from\/ $X_1,\ldots,X_n$ are pairwise disjoint,
\item\label{item2:3} $h(X_{i_1}\cup\cdots\cup X_{i_t})\leq f(t)$ for any choice of\/ $i_1,\ldots,i_t\in\{1,\ldots,n\}$,
\end{enumerate}
then\/ $n<N$.
\end{lemma}

\begin{proof}
We proceed by induction on $k$.
For $k=1$, we can set $N_f(1,\ell)=\ell$ and the assertion obviously holds.
For $k\geq 2$, let $H$ denote the $k$-uniform hypergraph with $V(H)=\{X_1,\ldots,X_n\}$ and $E(H)$ consisting of the $k$-tuples of sets with nonempty intersection.
Let $m=N_f(k-1,\ell)$ and $\alpha=1/\tbinom{m}{k}$.
Set $N_f(k,\ell)=\max\{M_f(k,\alpha),m\}$ for $M_f$ as claimed by Lemma~\ref{lem:ind1}.
If $|E(H)|\geq\alpha\tbinom{n}{k}$, then we can apply Lemma~\ref{lem:ind1} to conclude that $n\leq N_f(k,\ell)$.
Thus, suppose $|E(H)|<\alpha\tbinom{n}{k}$.
Since $n\geq m$, it follows from Theorem~\ref{thm:turan} that $H$ contains an independent set $\calI$ of size $m$.
Such an independent set is a $(k-1)$-fold packing, so we can apply the induction hypothesis to $\calI$ and conclude that $m<N_f(k-1,\ell)$, which is a contradiction.
\end{proof}

\begin{proof}[Proof of Theorem~\ref{thm:main2}]
Let $\calX$ be a finite subfamily of $\calF$ that forms a $k$-fold packing, and let $G$ be the intersection graph of $\calX$.
We show, for a suitable constant $t\in\setN$ which depends on $k$ and $f$, that $G$ contains no subgraph isomorphic to $K_{t,t}$.
Then, by Theorem~\ref{thm:fox-pach}, $G$ contains a vertex of degree smaller than $c(t)$, so that we can set $p_f(k)=c(t)$.

Suppose, for some $\ell\in\setN$, that $G$ contains an \emph{induced} subgraph isomorphic to $K_{\ell,\ell}$, and let $\calY\subseteq\calX$ be the set of vertices of this subgraph.
By \ref{item2:3}, we have $f(2\ell)\geq h(\bigcup\calY)$.
By Lemma~\ref{lem:euler}, the fact that $\calY$ is a $2$-fold packing, the remark after the proof of Lemma~\ref{lem:euler}, and the assumption that each set in $\calF$ is arc-connected, we have $h(\bigcup\calY)\geq\ell^2-2\ell+1$.
Hence we have $f(2\ell)\geq\ell^2-2\ell+1$, which contradicts the assumption that $f(n)=o(n^2)$ if $\ell$ is large enough.
Therefore, we can assume that $G$ contains no induced $K_{\ell,\ell}$.

Let $t=N_f(k,\ell)$ for $N_f$ as claimed by Lemma~\ref{lem:ind2}.
Suppose for a contradiction that $G$ contains a subgraph isomorphic to $K_{t,t}$.
Let $\calA$ and $\calB$ denote its two vertex classes.
At least one of $\calA,\calB$, say $\calA$, contains no independent set (packing) of size $\ell$, as otherwise the two independent sets, one in $\calA$ and one in $\calB$, would induce a subgraph isomorphic to $K_{\ell,\ell}$ in $G$.
Therefore, the assumptions of Lemma~\ref{lem:ind2} are satisfied for $\calA$, and we conclude that $|\calA|<t$.
This contradiction completes the proof of Theorem~\ref{thm:main2}.
\end{proof}

\section*{Acknowledgment}

\noindent
We thank Rado\v{s} Radoi\v{c}i\'c and anonymous referees for their valuable remarks and suggestions.


\begin{thebibliography}{99}

\bibitem{AsG60}
Edgar Asplund and Branko Gr\"unbaum, On a colouring problem, \emph{Math. Scand.} 8:181--188, 1960.

\bibitem{Ber08}
Mark de Berg, Improved bounds on the union complexity of fat objects, \emph{Discrete Comput. Geom.} 40(1):127--140, 2008.

\bibitem{Blu57}
William~J. Blundon, Multiple covering of the plane by circles, \emph{Mathematika} 4(1):7--16, 1957.

\bibitem{Bur65}
James~P. Burling, On coloring problems of families of prototypes, Ph.D.\ thesis, University of Colorado, Boulder, 1965.

\bibitem{ClS89}
Kenneth~L. Clarkson and Peter~W. Shor, Applications of random sampling in computational geometry, II, \emph{Discrete Comput. Geom.} 4(5):387--421, 1989.

\bibitem{DuH72}
Vishwa~C. Dumir and Rajinder~J. Hans-Gill, Lattice double packings in the plane, \emph{Indian J. Pure Appl. Math.} 3(3):481--487, 1972.

\bibitem{Efr05}
Alon Efrat, The complexity of the union of $(\alpha,\beta)$-covered objects, \emph{SIAM J. Comput.} 34(4):775--787, 2005.

\bibitem{ERS93}
Alon Efrat, G\"unter Rote, and Micha Sharir, On the union of fat wedges and separating a collection of segments by a line, \emph{Comput. Geom.} 3:277--288, 1993.

\bibitem{EfS00}
Alon Efrat and Micha Sharir, On the complexity of the union of fat convex objects in the plane, \emph{Discrete Comput. Geom.} 23(2):171--189, 2000.

\bibitem{FeT72}
L\'aszl\'o Fejes T\'oth, \emph{Lagerungen in der Ebene, auf der Kugel und im Raum} (Arrangements in the plane, on the sphere, and in the space), \emph{Grundlehren Math. Wiss.}, vol.~65, Springer, Berlin, 1953.

\bibitem{FTK93}
G\'abor Fejes T\'oth and W{\l}odzimierz Kuperberg, A survey of recent results in the theory of packing and covering, in: J\'anos Pach (ed.), \emph{New Trends in Discrete and Computational Geometry}, \emph{Algorithms Combin.}, vol.~10, pp. 251--279, Springer, Berlin, 1993.

\bibitem{FoP10}
Jacob Fox and J\'anos Pach, A separator theorem for string graphs and its applications, \emph{Combin. Prob. Comput.} 19(3):371--390, 2010.

\bibitem{FoP14}
Jacob Fox and J\'anos Pach, Applications of a new separator theorem for string graphs, \emph{Combin. Prob. Comput.} 23(1):66--74, 2014.

\bibitem{Hel30}
Eduard Helly, \"Uber Systeme von abgeschlossen Mengen mit gemeinschaftlichen Punkten (On systems of closed sets with common points), \emph{Monatsh. Math. Phys.} 37(1):281--302, 1930.

\bibitem{Hep59}
Alad\'ar Heppes, Mehrfache gitterf\"ormige Kreislagerungen in der Ebene (Multiple lattice circle packings in the plane), \emph{Acta Math. Acad. Sci. Hungar.} 10(1--2):141--148, 1959.

\bibitem{KNS64}
Gyula Katona, Tibor Nemetz, and Mikl\'os Simonovits, \'Ujabb bizony\'{\i}t\'as a Tur\'an-f\'ele gr\'aft\'etelre \'es megjegyz\'esek bizonyos \'altal\'anos\'{\i}t\'asaira (Another proof of Tur\'an's graph theorem and remarks on its generalizations), \emph{Mat. Lapok} 15(1--3):228--238, 1964.

\bibitem{KLPS86}
Klara Kedem, Ron Livne, J\'anos Pach, and Micha Sharir, On the union of Jordan regions and collision-free translational motion amidst polygonal obstacles, \emph{Discrete Comput. Geom.} 1(1):59--71, 1986.

\bibitem{Kos04}
Alexandr Kostochka, Coloring intersection graphs of geometric figures with a given clique number, in: J\'anos Pach (ed.), \emph{Towards a Theory of Geometric Graphs}, \emph{Contemp. Math.}, vol.~342, pp. 127--138, AMS, Providence, 2004.

\bibitem{Kre98}
Marc van Kreveld, On fat partitioning, fat covering, and the union size of polygons, \emph{Comput. Geom.} 9(4):197--210, 1998.

\bibitem{MPS+94}
Ji\v{r}\'{\i} Matou\v{s}ek, J\'anos Pach, Micha Sharir, Shmuel Sifrony, and Emo Welzl, Fat triangles determine linearly many holes, \emph{SIAM J. Comput.} 23(1):154--169, 1994.

\bibitem{MiP}
Piotr Micek and Rom Pinchasi, Note on the number of edges in families with linear union-complexity, manuscript, \href{http://arxiv.org/abs/1312.1678}{arXiv:1312.1678}.

\bibitem{Pac80}
J\'anos Pach, Decomposition of multiple packing and covering, in: \emph{Diskrete Geometrie}, \emph{2. Kolloq. Inst. Math. Univ. Salzburg}, pp. 169--178, 1980.

\bibitem{PKK+14}
Arkadiusz Pawlik, Jakub Kozik, Tomasz Krawczyk, Micha{\l} Laso\'n, Piotr Micek, William~T. Trotter, and Bartosz Walczak, Triangle-free intersection graphs of line segments with large chromatic number, \emph{J. Combin. Theory Ser. B} 105:6--10, 2014.

\bibitem{PKK+13}
Arkadiusz Pawlik, Jakub Kozik, Tomasz Krawczyk, Micha{\l} Laso\'n, Piotr Micek, William~T. Trotter, and Bartosz Walczak, Triangle-free geometric intersection graphs with large chromatic number, \emph{Discrete Comput. Geom.} 50(3):714--726, 2013.

\bibitem{Sha91}
Micha Sharir, On $k$-sets in arrangements of curves and surfaces, \emph{Discrete Comput. Geom.} 6(1):593--613, 1991.

\bibitem{Spe72}
Joel Spencer, Tur\'an's theorem for $k$-graphs, \emph{Discrete Math.} 2(2):183--186, 1972.

\end{thebibliography}
\end{document}